\documentclass[10pt]{amsart}
\usepackage{amssymb}
\usepackage{amscd}
\usepackage{mathrsfs}
\usepackage[all]{xy}

\def\today{\number\day\space\ifcase\month\or   January\or February\or
   March\or April\or May\or June\or   July\or August\or September\or
   October\or November\or December\fi\   \number\year}

\theoremstyle{definition}
\newtheorem{thm}{Theorem}[section]
\newtheorem{lem}[thm]{Lemma}
\newtheorem{prp}[thm]{Proposition}
\newtheorem{dfn}[thm]{Definition}

\newtheorem{rmk}[thm]{Remark}

\newcommand{\beq}{\begin{equation}}
\newcommand{\eeq}{\end{equation}}
\newcommand{\beqr}{\begin{eqnarray*}}
\newcommand{\eeqr}{\end{eqnarray*}}
\newcommand{\bal}{\begin{align*}}
\newcommand{\eal}{\end{align*}}
\newcommand{\bei}{\begin{itemize}}
\newcommand{\eei}{\end{itemize}}

\newcommand{\Z}{{\mathbb{Z}}}

\newcommand{\N}{{\mathbb{N}}}

\newcommand{\Aut}{{\mathrm{Aut}}}

\newcommand{\Index}{{\mathrm{Index}}}

\title[The Rokhlin property for inclusions of C*-algebras]
{The Rokhlin property for inclusions of C*-algebras}
\author{Hiroyuki Osaka}
\date{\today}
\thanks{$^*$Research of the first author partially supported by the JSPS grant for Scientific Research No.23540256 
and No.17K05285}
\address{ Department of Mathematical Sciences\\
  Ritsumeikan University\\ Kusatsu, Shiga, 525-8577  Japan}

\email[]{osaka@se.ritsumei.ac.jp}

\author{Tamotsu Teruya}

\address{Faculty of Education, Gunma University, 4-2 Aramaki-machi,
Maebashi City, Gunma, 371-8510, Japan}

\email[]{teruya@gunma-u.ac.jp}

\keywords{Jiang-Su absorption, Inclusion of C*-algebras, strictly comparison}
\subjclass[2000]{Primary 46L55; Secandary 46L35.}

\begin{document}

\begin{abstract}
Let $P \subset A$ be an inclusion of $\sigma$-unital 
C*-algebras with a finite index in the sense of Izumi. 
Then we introduce the Rokhlin property for a conditional expectation 
$E$ from $A$ onto $P$
and show that 
if $A$ is simple and satisfies any of the property $(1) \sim (12)$
 listed in the below, 
 and $E$ has the Rokhlin property, then so does $P$.

\vskip 2mm

\begin{enumerate}
\item[(1)]
Simplicity;
\item[(2)]
Nuclearity;
\item[(3)]
C*-algebras that absorb a given strongly self-absorbing C*-algebra $\mathcal{D}$;
\item[(4)]
C*-algebras of stable rank one;
\item[(5)]
C*-algebras of real rank zero;
\item[(6)]
C*-algebras of nuclear dimension at most $n$, where $n \in \Z^+$;
\item[(7)]
C*-algebras of decomposition rank at most $n$, where $n \in \Z^+$;
\item[(8)]
Separable simple C*-algebras that are stably isomorphic to AF algebras;
\item[(9)]
Separable simple C*-algebras that are stably isomorphic to AI algebras;
\item[(10)]
Separable simple C*-algebras that are stably isomorphic to AT algebras;
\item[(11)]
Separable simple C*-algebras that are stably isomorphic to 
sequential direct limits of one dimensional NCCW complexes;
\item[(12)] 
Separable C*-algebras with strict comparison of positive elements.
\end{enumerate}

%\color{blue}
In particular, when $\alpha : G \rightarrow \rm{Aut}(A)$ 
is an action of a finite group $G$ on $A$ with the Rokhlin property  in the sense of Nawata, the properties $(1) \sim (12)$ are inherited to the fixed point algebra $A^\alpha$ and the crossed product algebra $A \rtimes_\alpha G$ from $A$.
\end{abstract}

\maketitle

%%%%%%%%%%%%%%%%%%%%%%%%%%%%%%%%%%%%%%%%%%%%%%%%%%%%%%%%%%%%%%%%%%%%%%%%
%\color{black}

\section{Introduction}

The Rokhlin property for finite group actions on $\sigma$-unital C*-algebras was introduced by Nawata \cite{NN} and was formulated by  using Kirchberg's central sequence C*-algebras $F(A) = A' \cap A^\infty/A_{nn}(A, A^\infty)$ as follows:
an action $\alpha$ of a finite group $G$ on a $\sigma$-unital C*-algebra 
$A$ is said to have the Rokhlin property if there exists a partition of unity $\{e_g\}_{g\in G} \subset F(A)$ consisting of projections satisfying $\alpha_g(e_h) = e_{gh}$  
for any $g, h \in G$. If $A$ is unital, then the definition above coincides with the definition of \cite{Izumi:Rokhlin}.

After that Santiago \cite{LS} showed that the Nawata's definition is equivalent to the generalized Rokhlin property in the case of a separable C*-algebra, that is, an action $\alpha$ of a finite group $G$ on a separable C*-algebra is said to have the generalized Rokhlin property if for any $\varepsilon > 0$ and any finite set $F \subset A$ there exist mutually orthogonal 
positive contractions $(r_g)_{g\in G} \subset A$ such that 
\begin{enumerate}
\item
$\|\alpha_g(r_g) - r_{gh}\| < \varepsilon$ for all $g, h \in G$;
\item
$\|r_ga - ar_g\| < \varepsilon$ for all $a \in F$ and $g \in G$;
\item
$\|(\sum_gr_g)a - a\| < \varepsilon $ for all $a \in F$.
\end{enumerate}
and showed that the permanence properties of C*-algebra $A$, like pure infiniteness, stable rank one, real rank zero, the nuclear dimension $n$, $\mathcal{D}$-absorption for a strongly self-absorbing C*-algebra $\mathcal{D}$, simplicity, $AF$, $AI$, $AT$-properties, the strict comparison property for Cuntz semigroup etc.,   are preserved under taking crossed products by 
actions of finite groups with the generalized Rokhlin property.

As continuous works in \cite{OT0} and \cite{OT} the authors introduce the Rokholin property for an inclusion of 
$\sigma$-unital C*-algebras $P \subset A$ with finite index in the sense of Izumi \cite{Izumi:inclusion} and show in Theorem~3.2 that several permanence properties for $A$ are preserved for $P$ using the sequentially split *homomorphisms technique by Barlak and Szab\'o \cite{BS}.
%\color{blue}
In particular, when $\alpha:G \rightarrow \rm{Aut}(A)$ is an action of a finite group $G$ on $A$, we show that $\alpha$ has the Rokhlin property in the sense of Nawata \cite{NN} if and only if the canonical conditional expectation $E:A \rightarrow A^\alpha$ has the Rokhlin property (Proposition~\ref{prp:Nawata}), hence several permanence properties for $A$ are inherited to the fixed point algebra $A^\alpha$ and the crossed product algebra $A \rtimes_\alpha G$. 

%\color{black}

%%%%%%%%%%%%%%%%%%%%%%%%%%%%%%%%%%%%%%%%%%%%%%%%%%%
\section{Rokhlin property}

At first we  recall Index for $\sigma$-unital C*-algebras in the sense of Izumi \cite{Izumi:inclusion}.

Let $P \subset A$ be an inclusion of $\sigma$-unital C*-algebras and $E$ be a faithful conditional expectation from $A$ onto $P$.
Then $E$ is of index finite in the sense of Izumi if 
$$
\sup\{\lambda > 0\colon \dfrac{1}{\lambda}E - \mathrm{Id} \ \hbox{is positive }\}.
$$
is finite. 
We define 
$\mathrm{Index}_p(E)  = \left(\sup\{\lambda > 0\colon \dfrac{1}{\lambda}E - \mathrm{Id} \ \hbox{is positive }\}\right)^{-1}$.
When there is no such number $\lambda$ satisfying the condition, we set $\mathrm{Ind}_p E = \infty$.

Note that when $A$ and $P$ are unital, $\mathrm{Index}_pE < \infty$ implies that there is a quasi-basis 
$\{(u_i, u_i^*)\}_{i=1}^n$ of $A \times A$ for $E$ such that for any $x \in A$ 
\begin{align*}
x = \sum_{i=1}^nu_iE(u_i^*x) = \sum_{i=1}^nE(xu_i)u_i^*,\\
\mathrm{Index}_p E \leq \sum_{i=1}^nu_iu_i^*  (= \mathrm{Index}_w E).  
\end{align*}

\vskip 3mm

%\color{red}

\begin{rmk}\label{dual conditional expectation}
Under the condition in Definition~\ref{dfn:Rokhlin} for an inclusion of $\sigma$-unital C*-algebras $P \subset A$, 
if $A$ is simple, $\widehat{E^{**}}(1)$ is scalar, that is, $\widehat{E^{**}}(1) \in Z(M(A))$ by \cite[Theorem~3.2]{Izumi:inclusion}.  
Then there is  the dual conditional expectation $E_1\colon A_1 \rightarrow A$ satisfies 
$E_1(a_1e_Pa_2) = \widehat{E^{**}}(1)^{-1}a_1a_2$ 
for any $a_1, a_2\in A$ by \cite[Theorem~2.8(1)]{Izumi:inclusion}, where $A_1 = C^*(\{ae_Pb\mid a, b \in A\})$. 
Moreover, since $\Index_pE_1 < \infty$, we know that there is a basic construction 
$A \subset A_1 \subset A_2$ by \cite[Corollary~3.4]{Izumi:inclusion}. 
We write $\mathrm{Index}_w E = \widehat{E^{**}}(1)$.

(see \cite[Definition~2.3 and Theorem~2.8]{Izumi:inclusion}.)  
\end{rmk}

%\color{black}

\vskip 3mm

For a C*-algebra $A$, we set
\begin{align*}
C_0(A) &= \{(a_n) \in \ell^\infty(\N, A)\colon \lim_{n\rightarrow\infty}\|a_n\| = 0\}\\
A^\infty &= \ell^\infty(\N, A)/C_0(A).
\end{align*}

We identify $A$ with the C*-subalgebra of $A^\infty$ consisting of the equivalence classes
of constant sequences.

\vskip 3mm

\begin{dfn}\label{dfn:Rokhlin}
Let $P \subset A$ be an inclusion of $\sigma$-unital C*-algebras and $E\colon A \rightarrow P$ be 
a conditional expectation of index finite in the sense of Izumi. 
A conditional expectation $E$ is said to have the {\it Rokhlin property} 
if there exists a contractive positive element  $e \in A^{'} \cap A^\infty$ satisfying 
\begin{itemize}
\item[(i)]
$[e] \in N(\overline{AA^\infty A}, A^\infty)/A_{nn}(\overline{AA^\infty A}, A^\infty)$ is a projection,
\item[(ii)]
$$
({\Index}_wE)[E^\infty (e)] = 1,
$$
\item[(iii)]
A map 
$A \ni x \mapsto  [xe] \in N(\overline{AA^\infty A}, A^\infty)/A_{nn}(\overline{AA^\infty A}, A^\infty)$
is injective, 
\end{itemize}
where 
$N(\overline{AA^\infty A}, A^\infty) = \{x \in A^\infty \mid x\overline{AA^\infty A} + \overline{AA^\infty A}x  
\subset \overline{AA^\infty A}\}$, 
$Ann(\overline{AA^\infty A}, A^\infty) = \{x \in A^\infty| xa = ax = 0\ \hbox{for all}\ a \in \overline{AA^\infty A}\}$
, $E^\infty$ is the conditional expectation from $A^\infty$ onto $P^\infty$ defined by $E^\infty(x) = 
(E(x_n))$ for $x = (x_n)\in A^\infty$.
We call $e$ a Rokhlin element.
\end{dfn}

\vskip 3mm

\begin{rmk}\label{rmk:basic}
Let $A$ be a $\sigma$-unital C*-algebra. Then we have 
\begin{enumerate}
\item
$A \subset \overline{AA^\infty A}$,
\item
$A_{nn}(A, A^\infty) \subset A^{'} \cap A^\infty \subset N(\overline{AA^\infty A}, A^\infty)$.
Note that $A_{nn}(A, A^\infty) = A_{nn}(\overline{AA^\infty A}, A^\infty)$.
\item
If $A$ is unital, then $N(\overline{AA^\infty A}, A^\infty) = A^\infty$ and 
$A_{nn}(A, A^\infty) = \{0\}$. Hence, the definition of the Rokhlin property is equvalent to 
that in \cite[Definition~2.2]{OT}.
\item 
We have $\mathrm{Index}_p E^\infty = \mathrm{Index}_p E$ since  
$\dfrac{1}{\lambda}E^\infty - \mathrm{Id} \ \hbox{is positive }$ if and only if $ \dfrac{1}{\lambda}E - \mathrm{Id} \ \hbox{is positive }$.
\end{enumerate}
\end{rmk}

\vskip 3mm

%\begin{rmk}\label{dual conditional expectation}
%Under the condition \textcolor{red}{in Definition~\ref{dfn:Rokhlin} for an inclusion of C*-algebras $P \subset A$}, 
%if $A$ is simple, $\widehat{E^{**}}(1)$ is scalar, that is, $\widehat{E^{**}}(1) \in Z(M(A))$ by 
%\cite[Theorem~3.2]{Izumi:inclusion}.  
%Then there is  the dual conditional expectation $E_1\colon A_1 \rightarrow A$ satisfies 
%$E_1(a_1e_Pa_2) = \widehat{E^{**}}(1)^{-1}a_1a_2$ 
%for any $a_1, a_2\in A$ by \cite[Theorem~2.8(1)]{Izumi:inclusion}, \textcolor{red}{where $A_1 = C^*(\{ae_Pb\mid a, b \in %A\})$}. 
%Moreover, since $\Index_pE_1 < \infty$, we know that there is a basic construction 
%$A \subset A_1 \subset A_2$ by \cite[Corollary~3.4]{Izumi:inclusion}. 
%\textcolor{red}{We write $\mathrm{Index}_w E = \widehat{E^{**}}(1)$}.
%\end{rmk}

\vskip 3mm

\begin{lem}\label{lem:relation}
Let $P \subset A$ be an inclusion of $\sigma$-unital C*-algebras 
and $E$ a conditional expectation from $A$ onto $P$
having $\Index_p E < \infty$ and the Rokhlin property
with  a Rokhlin element $e$, 
and let $e_p$ be the Jones projection for $E$. Suppose that $A$ is simple.
Then 
$$
(\Index_w E) [ee_pe] = [e]
$$
in $N(D(P, A_1^\infty), A_1^\infty)/A_{nn}(D(P, A_1^\infty),  A_1^\infty)$, where $D(P, A_1^\infty) = \overline{PA_1^\infty P}$.
\end{lem}

\vskip 3mm

\begin{proof}
Since $A_{nn}(D(P, A_1^\infty), A_1^\infty) = A_{nn}(P, A_1^\infty)$, we have only to show that for any $a \in P$
\begin{align*}
&a(\Index_w E)ee_Pe = ae \\
&(\Index E)ee_Pe a = ea.
\end{align*}

Note that for any $a \in A$ $ae^2 = ae$ since $A \subset D(A, A^\infty)$ by Remark~\ref{rmk:basic} (1).

Set $f = (\Index_w E)ee_Pe$. Then, for any $a \in P$
\begin{align*}
af^2 &= a(\Index_w E)^2(ee_Pe)(ee_Pe)\\
&= (\Index_w E)^2(ee_Pae^2e_Pe) \\
&= (\Index_w E)^2(ee_Paee_Pe) \ (ae^2 = ae)\\
&=(\Index_w E)^2(eaE^\infty(e)e_Pe)\\
&= a(\Index_w E)ee_Pe\\
&= af.
\end{align*}

Similary, $f^2a = fa$ for any $a \in P$.
Hence $[f]^2 = [f]$ in $N(D(P, A_1^\infty)$, $A_1^\infty)/A_{nn}(D(P,A_1^\infty),A_1^\infty)$.

Note that 
\begin{align*}
\hat{E}^\infty(e - f) &= e - (\Index_w E)\hat{E}^\infty(ee_Pe)\\
&=e - e = 0.
\end{align*}

Since $f \leq e^2$, there exists $x \in N(D(P,A_1^\infty)$, $A_1^\infty)$ such that $[e -f] = [x^*x]$. 
Then, for any $a \in P$ $a(e-f) = a(x^*x)$. Therefore,
%\begin{align*}
%0 = a\hat{E}^\infty(e-f) &= a\hat{E}^\infty(x^*x)\\
%&= 
$[\hat{E}^\infty(e-f)] = 0$ and $\dfrac{1}{\lambda}\hat{E}^\infty - \mathrm{Id} \ \hbox{is positive }$ for som $\lambda >0$,  we have 
$$
0 \leqq a \left( \dfrac{1}{\lambda}\hat{E}^\infty(e-f) - (e-f)   \right)a^* = - a(e-f)a^*.= -a(x^*x)a^*
$$
for any $a \in P$. 
Therefore $a(e-f) = 0 = (e-f)a$ and 
$[e] = [f]$, that is, $(\Index_w E)[ee_Pe] = [e]$.
\end{proof}

\vskip 3mm

%\begin{prp}\label{simplicity}
%Let $E\colon A \rightarrow P$ be a faithful conditional expecation with $\Index_pE < \infty$ and 
%the Rokhlin property, and let $e_p$ be the Jones 
%projection for $E$. Suppose that $A$ is simple. Then $P$ is simple.
%\end{prp}

%\vskip 3mm

\begin{lem}\label{lem:D}
Let $P \subset A$ be an inclusion of $\sigma$-unital C*-algebras and $E$ be a conditional 
expectation from $A$ onto $P$ with $\Index_p E < \infty$. 
Suppose $E$ has the Rokhlin property with a Rokhlin element $e$.

%前のバージョンだと$e$が何なのかよくわからないので変えました。
%Let $P \subset A$ be an inclusion of $\sigma$-unital C*-algebras with $\Index_p E < \infty$.
Then, for $x \in A$,  $E^\infty(xe) \in D(P, P^\infty)$. 
\end{lem}

\vskip 3mm

\begin{proof}
Let $x \in A$ and $x \geq 0$. 
Then 
$0 \leq E^\infty(xe) = E^\infty(x^{1/2}ex^{1/2}) \leq E^\infty(x) = E(x)$. 
By \cite[Proposition~1.4.5]{GP} there is a $u \in P^\infty$ and $0 < \alpha < 1/2$ such that 
$E^\infty(xe)^{\frac{1}{2}} = uE(x)^\alpha$.  Hence $E^\infty(xe) = E(x)^\alpha uu^* E(x)^\alpha \in \overline{PP^\infty P} = D(P, P^\infty)$.
\end{proof}

\vskip 3mm

The following implies that the inclusion map $P \subset A$ is a sequentially-split *-homomorphism in the sense of 
Barlak and Szab\'o \cite{BS}.

\vskip 3mm

\begin{lem}\label{lem:basic}
Let $P \subset A$ be an inclusion of $\sigma$-unital C*-algebras and $E$ be a conditional 
expectation from $A$ onto $P$ with $\Index_p E < \infty$. 
If $A$ is simple and $E$ has the Rokhlin property with a Rokhlin element $e \in A^\infty$, then there is  
a *-homomorphism $\beta :A \rightarrow P^\infty$ %N(D(P, P^\infty), P^\infty)/A_{nn}(P, P^\infty)$
such that $\beta(x)  = x$ for all $x \in P$.
\end{lem}

\begin{proof}
Let $e_p$ be the Jones projection for the inclusion $P \subset A$.
Then, for any $x \in A$ 

and any $a \in P$
\begin{align*}
axe &= \Index_w E)\hat{E}^\infty(e_Pxae)\\
&= (\Index_w E)^2\hat{E}^\infty(e_Pxaee_Pe)\\
&= (\Index_w E^2\hat{E}^\infty(aE^\infty(xe)e_Pe)\\
&= a(\Index_w E)E^\infty(xe)e,
\end{align*}
where $\hat{E}$ is the dual conditional expectation for $E$.
Put $y = (\Index_w E)E^\infty(xe)\in N(D(P, P^\infty), P^\infty)$ by Lemma~\ref{lem:D}. 
Then we have $axe = aye$.for any $a \in P$. 
Similarly, we have $xea = yea$ for any $a \in P$. Hence, $[xe] =[ye]$ in 
$N(D(P, A^\infty), A^\infty)/A_{nn}(D(P,A^\infty),A^\infty)$.

Suppose that $[ye] = [ze]$ for $y, z \in P^\infty$. Then
$$
[z] = [z](\Index_w E)[E^\infty(e)] = (\Index_w E)[E^\infty(ze)] = (\Index_w E)[E^\infty(ye)] = [y].
$$
Therefore, we obtain the uniqueness of $[y]$ as follows. Set $\rho(x) = [y]$.
Then $\rho$ is a linear map from $A$ to $N(D(P, P^\infty), P^\infty)/A_{nn}(P, P^\infty)$. 
In particular, $\rho(x) = [x]$ for all $x \in P$. 

Homomorphism property of $\rho$: Let $x_1, x_2 \in A$ and $\rho(x_1), \rho(x_2) \in 
N(D(P, P^\infty), P^\infty)/A_{nn}(P, P^\infty)$ 
such that $[x_1e] =\rho(x_1)[e]$, $[x_2e] = \rho(x_2)[e]$.

The we have
\begin{align*}
[x_1x_2e] &= [x_1ex_2]\\
&= \rho(x_1)[ex_2]\\
&= \rho(x_1)[x_2e]\\
&= \rho(x_1)\rho(x_2)[e]. \\
\end{align*}

 From the uniqueness, we have $\rho(x_1x_2) = \rho(x_1)\rho(x_2)$. The *-preserving property for $\rho$ can be proved 
as in the same step.

Since there is an isomorphosm $\pi$ from $N(D(P, P^\infty), P^\infty)/A_{nn}(P, P^\infty)$ to the multiplier algebra 
$M(D(P, P^\infty))$ of $D(P, P^\infty)$ by \cite[Proposition 1.5 (2)]{BS} such that 
$\pi(m)a = ma$, $a\pi(m) = am$ for $m \in N(D(P, P^\infty), P^\infty)$ and $a \in D(P, P^\infty)$, a map $\beta = \pi \circ \rho$ is a *-homomorphism from $A$ to $M(D(P, P^\infty))$ such that $\beta(x) = x$ for all $x \in P$.

Since for $x \in A$  $\beta(x) = \pi([E^\infty(xe)]) \in D(P, P^\infty) \subset P^\infty$, $\beta$ is a required *-homomorphism.
\end{proof}

\vskip 3mm
%%%%%%%%%%%%%%%%%%%%%%%%%%%%%%%%%%%%%%%%%%%%%%%%%%%%%%%%%%%%%%%%%%%%%%%%%%%%
\section{Main Result}

%\color{red}

Recall that for $\sigma$-unital C*-algebras $A$ and $B$ a *-homomorphism $\varphi\colon A \rightarrow B$ is called sequetially-split if there exists a *-homomorphism $\psi\colon B \rightarrow A^\infty$ such that $(\psi \circ \varphi)(a) = a$ for $a \in A$ \cite{BS}.

%\color{black}

\begin{thm}\label{thm:sequentially}
Let $P \subset A$ be an inclusion of $\sigma$-unital C*-algebras and 
$E:A \rightarrow P$ be a faithful conditional expectation with $\Index_p E < \infty$.
Suppose that $A$ is simple. Then an inclusion map $P \rightarrow A$ is a sequentially-split *-homomorphism. % in the sense of Barlak and Szab\'o.
\end{thm}

\vskip 3mm

\begin{proof}
Let $\iota\colon P \rightarrow A$ be an inclusion map.
From Lemma~\ref{lem:basic} 
there exists the map $\beta\colon A \rightarrow P^\infty$ such that 
$\beta \circ \iota (x) = x$ for $x \in P$. This means that the map $\iota$ is a sequentially-split *-homomorphism.
\end{proof}

\vskip 3mm

\begin{thm}\label{thm:main theorem}
Let $P \subset A$ be an inclusion of $\sigma$-unital C*-algebras and $E$ be a conditional expectation from $A$ onto $P$ with $\mathrm{Index}_pE < \infty$. If $A$ is simple and satisfies any of the property $(1) \sim (12)$
 listed in the below, 
 and $E$ has the Rokhlin property, then so does $P$.

\vskip 2mm

\begin{enumerate}
\item[(1)]
Simplicity;
\item[(2)]
Nuclearity;
\item[(3)]
C*-algebras that absorb a given strongly self-absorbing C*-algebra $\mathcal{D}$;
\item[(4)]
C*-algebras of stable rank one;
\item[(5)]
C*-algebras of real rank zero;
\item[(6)]
C*-algebras of nuclear dimension at most $n$, where $n \in \Z^+$;
\item[(7)]
C*-algebras of decomposition rank at most $n$, where $n \in \Z^+$;
\item[(8)]
Separable simple C*-algebras that are stably isomorphic to AF algebras;
\item[(9)]
Separable simple C*-algebras that are stably isomorphic to AI algebras;
\item[(10)]
Separable simple C*-algebras that are stably isomorphic to AT algebras;
\item[(11)]
Separable simple C*-algebras that are stably isomorphic to 
sequential direct limits of one dimensional NCCW complexes;
\item[(12)] 
Separable C*-algebras with strict comparison of positive elements.
\end{enumerate}

\end{thm}

\vskip 3mm

\begin{proof}

It follows from Theorem~\ref{thm:sequentially} and \cite[Theorem~2.9]{BS}.
\end{proof}

\vskip 3mm

Using Theorem~\ref{thm:main theorem} we can prove the permanence property for crossed products $A \rtimes_\alpha G$  by actions $\alpha$ from finite groups $G$ on a $\sigma$-unirtal simple C*-algebra $A$ with the Rokhlin property in the sense of Nawata \cite{NN} as below.

\vskip 3mm

\begin{dfn}\cite[Definition~3.1]{NN}\label{dfn:Nawata}
An action $\alpha$ of a finite group $G$ on a $\sigma$-unital C*-algebra $A$ is said
to have the Rokhlin property if there exists a partition of unity $\{e_g\}_{g \in G } \subset  A^{'} \cap A^\infty/A_{nn}(A, A^\infty)$
consisting of projections satisfying $\alpha_g(e_h) = e_{gh}$ for any $g, h \in G$.
A family of projections $\{e_g\}_{g\in G}$ is called Rokhlin projections of $\alpha$.
\end{dfn}

\vskip 3mm

When $A$ is separable, Definiyion~\ref{dfn:Nawata} is equivalent to the following.

\vskip 3mm

\begin{dfn}\cite[Corollary~2]{LS}\label{dfn:Santiago}
Let $A$ be a $\sigma$-unital C*-algebra  and let $\alpha:G \rightarrow \Aut(A)$ be an action of a finite group $G$ on 
$A$. An action $\alpha$ has the Rokhlin property if for any $\varepsilon > 0$ and any finite subset $F \subset A$ 
there exist mutually orthogonal positive contractions $(r_g)_{g\in G} \subset A$ such that 
\begin{enumerate}
\item[(i)] $||\alpha_g(r_h) - r_{gh}|| < \varepsilon$, for all $g,h \in G$;
\item[(ii)] $||r_ga - ar_g|| < \varepsilon$, for all $a \in F$ and $g \in G$;
\item[(iii)] $||(\sum_{g\in G}r_g)a - a|| < \varepsilon$, for all $a \in F$.
\end{enumerate}
\end{dfn}

\vskip 3mm

\begin{prp}\label{prp:Nawata}
Let $A$ be $\sigma$-unital simple C*-algebra  and let $\alpha:G \rightarrow \Aut(A)$ be an action of a finite group $G$ on 
$A$. 
Then  $\alpha$ has the Rokhlin property in the sense of Nawata if and only if 
the canonical conditional expectation $E:A \rightarrow A^\alpha$ has the Rokhlin property in the sense  of 
Definition~\ref{dfn:Rokhlin}.
\end{prp}

\begin{proof}
Let $\{e_g\}_{g_\in G}$ be the Rokhlin projections of $\alpha$. Then by \cite[Lemma~10.1.12]{TL}
these elements can be lifted to mutually orthogonal positive contractions $\{f_g\}_{g\in G} \in A^{'} \cap A^\infty$. 
Set $e = f_0$. Then $e \in A^{'} \cap A^\infty$ such that 
\begin{align*}
E^\infty(e) + A_{nn}(A, A^\infty) &= \frac{1}{|G|}\sum_{g\in G}\alpha_g(f_0)  + A_{nn}(A, A^\infty)\\
&= \frac{1}{|G|}\sum_{g\in G}e_g = \frac{1}{|G|}.
\end{align*}

Since $A^{'} \cap A^\infty/A_{nn}(A, A^\infty)$ is unital isomorphic into  
$N(D(A, A^\infty), A^\infty)/A_{nn}(D(A, A^\infty)$, $A^\infty)$ by \cite[Proposition~1.5(2)]{BS}, 
we have $[E^\infty(e)] = \frac{1}{|G|}$ in $N(D(A, A^\infty), A^\infty)/A_{nn}(D(A, A^\infty)$.
Since $A$ is simple, the third condition $(iii)$ in  Definition~\ref{dfn:Rokhlin} is automatically hold. Therefore, $E$ has the Rokhlin property in the sense of Definition~\ref{dfn:Rokhlin}.

%\color{red}
Conversely, $E$ has Rokhlin  property in the sense of Definition~\ref{dfn:Rokhlin}. 
Then there is a contractive positive element $e \in A' \cap A^\infty$ such that 
$[e] \in N(\overline{AA^\infty A}, A^\infty)/A_{nn}(\overline{AA^\infty A}, A^\infty) 
=  N(\overline{AA^\infty A}, A^\infty)/A_{nn}(A, A^\infty) $ is a projection.
Set $e_g = \alpha_g([e])$ for each $g \in G$. Then $\{e_g\}_{g\in G}$ is a set of orthogonal projections 
in $A' \cap A^\infty/A_{nn}(A, A^\infty)$ such that 
$\sum_{g\in G}e_g = 1$.
\end{proof}

%\color{black}
%%%%%%%%%%%%%%%%%%%%%%%%%%%%%%%%%%%%%%%%%%%%%%%%%%%%%%%%%%%%%%%%%%%%%%%%%%%
%\color{blue}

\begin{rmk}
Under the same assumption in Proposition \ref{prp:Nawata} we know that an each property in $(1) \sim  (12)$ is inherited to the fixed point algebra $A^\alpha$ by Theorem \ref{thm:main theorem}. Moreover, 
since $A^\alpha$ is stably isomorphic to the crossed product algebra $A \rtimes_\alpha G$, an each property $(1) \sim  (12)$ in $A$ is inherited to $A \rtimes_\alpha G$, too.
\end{rmk}

%\color{black}

%\color{black}
%%%%%%%%%%%%%%%%%%%%%%%%%%%%%%%%%%%%%%%%%%%%%%%%

\end{document}